\documentclass[11pt]{article}

\usepackage{mathtools,amsmath,amsthm,amssymb,enumitem,dsfont,microtype}
\usepackage{color}
\usepackage[colorlinks]{hyperref}
\usepackage[margin=1.22in]{geometry}
\usepackage{scalerel}

\usepackage[T1]{fontenc}
\usepackage{mathpazo}

\usepackage[dvipsnames]{xcolor}
\newcommand\myshade{85}
\colorlet{mylinkcolor}{red}
\colorlet{mycitecolor}{blue}
\colorlet{myurlcolor}{Aquamarine}

\hypersetup{
  linkcolor  = mylinkcolor!\myshade!black,
  citecolor  = mycitecolor!\myshade!black,
  urlcolor   = myurlcolor!\myshade!black,
  colorlinks = true,
}

{
\makeatletter
\addtocounter{footnote}{1} 
\renewcommand\thefootnote{\@fnsymbol\c@footnote}%
\makeatother
}

\newcommand\blfootnote[1]{%
  \begingroup
  \renewcommand\thefootnote{}\footnote{#1}%
  \addtocounter{footnote}{-2}%
  \endgroup
}

\providecommand{\keywords}[1]{\noindent\textbf{Keywords:} #1}

\newcommand{\RR}{\mathbb{R}}

\newcommand{\ra}{\rightarrow}
\newcommand{\norm}[1]{\left\|#1\right\|}
\newcommand{\half}{{\scaleobj{.8}{\frac{1}{2}}}}

\newcommand{\eg}{{\it e.g.}}
\newcommand{\ie}{{\it i.e.}}

\DeclareMathOperator{\Tr}{Tr}

\DeclareMathOperator{\Diag}{Diag}
\DeclareMathOperator{\Rank}{Rank}
\DeclareMathOperator{\sgn}{sgn}

\newtheorem{lem}{Lemma}
\newtheorem{thm}{Theorem}
\newtheorem{coro}{Corollary}

\newtheorem{prop}{Proposition}

\begin{document}
\title{A Matrix Generalization of the Hardy-Littlewood-P\'olya Rearrangement Inequality and Its Applications\blfootnote{Communicated by Defeng Sun.}}
\author{Man-Chung Yue\thanks{Musketeers Foundation Institute of Data Science and Department of Industrial and Manufacturing Systems Engineering, The University of Hong Kong, Hong Kong, China. E--mail: {\tt mcyue@hku.hk}}}
\date{\today}
\maketitle

\begin{abstract} 
By analyzing an optimization problem over orthogonal matrices, we prove a generalization of the Hardy-Littlewood-P\'olya rearrangement inequality to positive definite matrices. The inequality is then extended to rectangular matrices. Using our main results, we derive new inequalities for several distance-like functions encountered in various signal processing or machine learning applications.
\end{abstract}

\bigskip

\section{Introduction}
The well-known Hardy-Littlewood-P\'olya rearrangement inequality \cite{hardy1952inequalities} states that for any vectors $u, v\in \mathbb{R}^n$,
\begin{equation}\label{ineq:HLP}
\sum_{i=1}^n u^\downarrow_i v^\uparrow_i \le \sum_{i=1}^n u_i v_i\le   \sum_{i=1}^n u^\downarrow_i v^\downarrow_i ,
\end{equation}
where $u^\downarrow$ and $v^\downarrow$ ($u^\uparrow$ and $v^\uparrow$) are the vectors with entries of $u$ and $v$ sorted in descending (ascending) order, respectively.
For positive vectors, a generalization of the rearrangement inequality~\eqref{ineq:HLP} is obtained in \cite{london1970rearrangement}, see also \cite[Example 3]{vince1990rearrangement}.
\begin{thm}[London~{\cite[Theorem 2]{london1970rearrangement}}]
\label{thm:London}
Let $u \in \mathbb{R}_{++}^n$,  $v\in \mathbb{R}_+^n$ and  $f : \mathbb{R}_+ \ra \mathbb{R}$ be any convex function such that $f(s) \ge f(0)$ for any $s \ge 0$. Then,
\begin{equation*}\label{ineq:London}
\sum_{i = 1}^n f (u_i^\downarrow v_i^\uparrow ) \le \sum_{i = 1}^n f \left(u_i v_i \right) \le \sum_{i = 1}^n f (u_i^\downarrow v_i^\downarrow ).
\end{equation*}
\end{thm}

There are also various generalizations of inequality~\eqref{ineq:HLP} to the matrix setting, where the entries of vectors are replaced by the eigenvalues or singular values of matrices. One such example is the following result. To state it, we denote the $i$-th largest eigenvalue by $\lambda_i(\cdot)$.
\begin{thm}[{Carlen and Lieb~\cite[Theorems 3.1-3.2]{carlen2006some}\footnote{In \cite[Theorems 3.1-3.2]{carlen2006some}, the more general inequalities
\begin{equation*}
\sum_{i = 1}^n \lambda_i^q (A) \lambda_{n - i +1}^{p+q} (B) \le \sum_{i=1}^n \lambda_i \big( B^p (B^{\half} A B^{\half})^q \big) 
\quad
\text{and} 
\quad
\sum_{i=1}^n \lambda_i \big( B^p (B^{\half} A B^{\half})^q \big) \le \sum_{i=1}^n \lambda_i^q (A) \lambda_i^{p+q} (B)
\end{equation*}
are obtained for $p \ge 0$ and under the same conditions on $A$, $B$, and $q$. However, as shown in the proof in~\cite{carlen2006some}, the first inequality can be reduced to the case of $p=0$. 
}}]
\label{thm:Carlen_Lieb}
Let $A,B\in \mathbb{R}^{n\times n}$ be positive semidefinite matrices and $q \ge 1$. Then, it holds that
\begin{equation*}
\sum_{i = 1}^n \lambda_i^q (A) \lambda_{n - i +1}^q (B) \le \sum_{i = 1}^n \lambda_i^q (B^{\half} A B^{\half})  .
\end{equation*}
If $q \ge 1$ is an integer, it also holds that
\begin{equation*}
\sum_{i=1}^n \lambda_i^q (B^{\half} A B^{\half}) \le \sum_{i=1}^n \lambda_i^q (A) \lambda_i^q (B).
\end{equation*} 
\end{thm}
The Hardy-Littlewood-P\'olya rearrangement inequality~\eqref{ineq:HLP} and its generalizations are useful tools in mathematical analysis and have found many applications in both pure and applied mathematics. They have been utilized in the studies of the geometry of Banach spaces~\cite{tomczak1974moduli,carlen2006some}, quantum entanglement~\cite{augusiak2008beyond}, covariance matrix estimation~\cite{rychener2024geometric} and  wireless communication~\cite{jorswieck2004performance,dorpinghaus2015log}, to name a few.

Our main contribution is the following matrix rearrangement inequality that generalizes both Theorem~\ref{thm:London} (up to differentiability requirement) and Theorem~\ref{thm:Carlen_Lieb}.

\begin{thm}
\label{thm:matrix_rearrange}
Let $f:\mathbb{R}_{++}\ra \mathbb{R}$ be a differentiable function such that $s\mapsto sf'(s)$ is monotonically increasing on $\RR_{++}$. Then, for any positive definite matrices $A,B\in \mathbb{R}^{n\times n}$,
\begin{equation*}
\sum_{i=1}^n f\left(\lambda_i (A)\lambda_{n-i+1} (B)\right) \le \sum_{i = 1}^n f\big( \lambda_i( B^{\half} A B^{\half} )\big) \le \sum_{i=1}^n f\left(\lambda_i (A)\lambda_i (B)\right).
\end{equation*}
If $f$ is additionally defined and right-continuous at $0$, then the inequality holds for any positive semidefinite matrices $A,B\in \mathbb{R}^{n\times n}$.
\end{thm}

The proof of Theorem~\ref{thm:matrix_rearrange} is based on the analysis of a certain optimization problem over orthogonal matrices and reveals that the matrices $A$ and $B$ commute at optimality. Specifically, note that the function $f$ induces a spectral function on the space of positive definite matrices, which defines $f(X) = U\Diag(f(\lambda_1(X)), \dots, f(\lambda_n(X)))V^\top$ for any positive definite matrix $X\in\mathbb{R}^{n\times n}$, where $U\Diag(\lambda_1(X),\dots, \lambda_n(X))V^\top$ is an eigenvalue decomposition of $X$ and $\Diag(\lambda_1(X),\dots, \lambda_n(X))$ is the diagonal matrix containing the eigenvalues of $X$ on its diagonal  sorted in descending order. Therefore, the middle sum can be written as $\Tr( f(B^{\half} A B^{\half}) )$, and Theorem~\ref{thm:matrix_rearrange} can be seen as a commutation principle for the function $X\mapsto \Tr(f(X))$ in the sense of \cite[Lemma~4]{iusem2007angular}, see also \cite[Theorem~7]{ramirez2013commutation} for the generalization of \cite[Lemma~4]{iusem2007angular} to continuously differentiable matrix functions. Nevertheless, \cite[Theorem~7]{ramirez2013commutation} is not directly applicable to our situation as the function $X\mapsto \Tr(f(X))$ is not continuously differentiable in general. Moreover, our proof is different from that of \cite[Theorem~7]{ramirez2013commutation} and highlights the importance of the monotonicity of $s\mapsto s f'(s)$ and positive definiteness of $A$ and $B$.

Theorem~\ref{thm:matrix_rearrange} does not only generalize Theorem~\ref{thm:London} from vector case to matrix case but also relaxes the condition on the function $f$ (up to differentiability requirement). To see this, consider a function $f$ satisfying the assumption of Theorem~\ref{thm:London}, which implies in particular that $f$ is defined and right-continuous at $0$. Suppose in addition that $f$ is differentiable on $\mathbb{R}_{++}$. Then, we have that for any $s > 0$,
\begin{equation*}
f (s) \ge f(0) \ge f(s) -s f'(s),
\end{equation*}
where the two inequalities follow from the assumption of Theorem~\ref{thm:London}. Therefore, $f' (s) \ge 0 $ for any $s > 0$. Hence, for any $s_2 > s_1 > 0$,
\begin{equation}\label{ineq:sf'_monotone}
s_2 f'(s_2) - s_1 f'(s_1) = (s_2 - s_1) f'(s_2) + s_1 (f'(s_2) - f'(s_1)) \ge 0  , 
\end{equation}
where we used the fact that $f'(s_2) \ge f'(s_1)$, due to the convexity of $f$.
This shows that the function $s \mapsto s f' (s)$ is monotonically increasing on $\mathbb{R}_{++}$. Furthermore, by taking $f(s) = s^q$ for $q>0$, one readily sees that Theorem~\ref{thm:matrix_rearrange} recovers Theorem~\ref{thm:Carlen_Lieb}. Note also that the requirement on $q$ is less stringent than Theorem~\ref{thm:Carlen_Lieb}.

The rest of the paper is organized as follows. Section~\ref{sec:prelim} prepares some auxiliary results. The main result Theorem~\ref{thm:matrix_rearrange} and its extension to rectangular matrices will be proved in Section~\ref{sec:main}. In Section~\ref{sec:applications}, we will present several applications of our matrix rearrangement inequalities. These applications are related to Schatten quasi-norms (see Section~\ref{subsec:schatten}), affine-invariant distance of positive definite matrices (see Section~\ref{subsec:FR_RG}) and Alpha-Beta log-determinant divergences (see Section~\ref{subsec:ABLD}).

\subsection{Notation}
The sets of non-negative and positive real numbers are denoted by $\mathbb{R}_+$ and $\mathbb{R}_{++}$, respectively.
For any vector $u\in \mathbb{R}^n$, the $n\times n$ diagonal matrix with the $i$-th diagonal entry given by $u_i$ is denoted by $\Diag(u)$. Also, we denote by $u^\downarrow$ and $u^\uparrow$ the vectors with entries of $u$ sorted in descending and ascending orders, respectively.
The sets of $n\times n$ symmetric matrices, positive definite matrices and orthogonal matrices are denoted by $\mathbb{S}_n$, $\mathbb{P}_n$ and $\mathbb{O}_n$, respectively.
For any matrix $X\in\RR^{m\times n}$, we denote by $\sigma(X) = (\sigma_1 (X), \dots, \sigma_{\min\{m,n\}}(X))^\top$ the vector of singular values sorted in descending order.
Also, for any $X\in\mathbb{S}_n$, we denote by $\lambda (X) = (\lambda_1(X),\dots,\lambda_n(X) )^\top$ the vector of eigenvalues sorted in descending order. Finally, the $n\times n$ identity matrix is denoted by $I_n$.

\section{Auxiliary Results}
\label{sec:prelim}

From now on, for any function $f:J \to \mathbb{R}$ with domain $J\subseteq \mathbb{R}_+$ and matrix $X\in \mathbb{R}^{m\times n}$ with singular values in $J$, we define 
\begin{equation*}\label{def:S_f_sing}
S_f (X) = \sum_{i = 1}^{\min\{m,n\}} f\left(\sigma_i(X) \right).
\end{equation*}
When $X\in \mathbb{R}^{n\times n}$ is a positive semidefinite matrix with eigenvalues in $J$, we have 
\begin{equation*}\label{eq:S_F_eig}
S_f (X) = \sum_{i = 1}^n f\left(\lambda_i(X) \right) .
\end{equation*}

We present a useful fact about the differentiability of the function $S_f$.

\begin{prop}[{Lewis and Sendov~\cite[Theorem~7.1 and Corollary~7.4]{lewis2005nonsmooth}}]\label{prop:pert}
Assume $m\le n$. Let $f:\mathbb{R}_{++} \ra \RR$ be a differentiable function and $X\in\mathbb{R}^{m\times n}$ be a full-rank matrix (\ie, $\Rank(X) = m$) with singular value decomposition $U\Diag(\sigma (X)\ 0) V^\top$. Then, $S_f$ is differentiable at $X$ with the derivative given by $U \left( \Diag\left(f'\left(\sigma(X)\right)\right)\   0 \right) V^\top$. 
\end{prop}
\noindent Note that the derivative is independent of the choice of singular value decomposition of $X$. Moreover, it is a symmetric matrix if $X \in \mathbb{P}_n$.

The proof of Theorem~\ref{thm:matrix_rearrange} also relies on the following three elementary lemmas, whose proofs are included for self-containedness. The first one is a vector rearrangement inequality.

\begin{lem}\label{lem:rearrange}
Let $f:\mathbb{R}_{++}\ra \mathbb{R}$ be a differentiable function such that $s\mapsto sf'(s)$ is monotonically increasing on $\RR_{++}$. Then, for any positive vectors $u,v\in \mathbb{R}^n$,
\begin{equation*}
\sum_{i=1}^n f(u^\uparrow_i v^\downarrow_i) \le \sum_{i=1}^n f(u_iv_i)\le   \sum_{i=1}^n f(u^\downarrow_i v^\downarrow_i).
\end{equation*}
\end{lem}

\begin{proof}
We prove only the second inequality as the first one can be proved using exactly the same argument.
By re-indexing the components of $v$, we can assume without loss of generality that $v_1 \ge \cdots \ge v_n$, \ie, $v^\downarrow = v$. Suppose that there exist indices $i,j\in \{1,\dots,n\}$ with $i<j$ such that $u_i \le u_j$. We claim that 
\[ f(u_i v_i^\downarrow) + f(u_j v_j^\downarrow) \le f(u_i v_j^\downarrow) + f(u_j v_i^\downarrow).\]
In other words, if there are two components $u_i$ and $u_j$ that are not sorted in descending order, then swapping them in the corresponding summands will not decrease the value of the sum.
Therefore, it suffices to prove that
\begin{equation*}
f(ac)+f(bd)-f(ad)-f(bc)\geq 0\quad \forall a\geq b> 0,\ c\geq d> 0.
\end{equation*} 
Towards that end, we define the function $g(t) = f(tc) - f(td)$ for $t >0$. By the assumption on $f$, 
\begin{equation*}
g'(t) = c f'(tc) - d f'(td) = \frac{1}{t} \left( tc f'(tc) - td f'(td) \right) \ge 0\quad \forall t>0.
\end{equation*}
Therefore, for any $c \ge d > 0$, $g$ is monotonically increasing and hence $g(a) \ge g(b)$, which is equivalent to the inequality $f(ac)+f(bd)-f(ad)-f(bc)\ge 0$. This completes the proof.
\end{proof}

The next one reveals the structure of matrices commuting with a diagonal matrix.

\begin{lem}\label{lem:diag_commute}
    Let $t_1,\dots, t_\ell \in \mathbb{R}$ be distinct real numbers, $n_1,\dots,n_\ell$ be positive integers and $X\in \mathbb{S}_n$ with $n = n_1 + \cdots + n_\ell$. Suppose that
    \begin{equation*}
        X \begin{pmatrix}
        t_1I_{n_1}&&\\
        &\ddots&\\
        &&t_\ell I_{n_\ell}\\
        \end{pmatrix} = \begin{pmatrix}
        t_1I_{n_1}&&\\
        &\ddots&\\
        &&t_\ell I_{n_\ell}\\
        \end{pmatrix} X.
    \end{equation*}
    Then, there exist symmetric matrices $\tilde{X}_1\in\mathbb{S}_{n_1},\dots, \tilde{X}_\ell \in \mathbb{S}_{n_\ell}$ such that
    \begin{equation*}
        X=\begin{pmatrix}
        \tilde{X}_1&&\\
        &\ddots&\\
        &&\tilde{X}_\ell
        \end{pmatrix}.
    \end{equation*}
\end{lem}

\begin{proof}
    To prove the lemma, for any matrix $Y\in \mathbb{R}^{n\times n}$, we partition its entries into blocks so that for any $i,j = 1,\dots, \ell$, the $ij$-th block, denoted by $[Y]_{ij}$, is $n_i \times n_j$.
    Let $D\in \mathbb{R}^{n\times n}$ be the diagonal matrix such that $[D]_{ii} = t_i I_{n_i}$.
    Then, for any $i,j = 1,\dots, \ell$, we have
    \[ [XD]_{ij} = \sum_{ k = 1}^\ell [X]_{ik} [D]_{kj} = t_j [X]_{ij} \quad\text{and}\quad [DX]_{ij} = \sum_{ k = 1}^\ell [D]_{ik} [X]_{kj} = t_i [X]_{ij} .\]
    The supposition implies $t_i [X]_{ij} = t_j [X]_{ij}$. Since $t_1,\dots, t_\ell$ are distinct, we conclude that $[X]_{ij}$ is a zero matrix for any $i \neq j$. This completes the proof.
\end{proof}

The last one concerns diagonal matrices in the same orthogonal conjugacy class. 

\begin{lem}
    \label{lem:orthogonal_conjugation}
    Let $D,\hat{D}\in \mathbb{R}^{n\times n}$ be two diagonal matrices and $Q\in \mathbb{O}_n$. Suppose that the diagonal entries of $D$ are distinct and $\hat{D} = Q D Q^\top $. Then, $Q$ is a permutation matrix.
\end{lem}

\begin{proof}
    By supposition, we have that $\hat{D}Q = Q D$. Next, for any $i,j = 1,\dots,n$,
    \begin{equation*}
        (\hat{D}Q)_{ij} = \sum_{k = 1}^n \hat{D}_{ik} Q_{kj} = \hat{D}_{ii} Q_{ij} \quad\text{and}\quad (QD)_{ij} = \sum_{k = 1}^n  Q_{ik} D_{kj} = D_{jj} Q_{ij},
    \end{equation*}
    which implies that $(\hat{D}_{ii} - D_{jj}) Q_{ij} = 0$. Since eigenvalues are preserved by conjugation, we know that the set of numbers on the diagonal of $\hat{D}$ must be the same as that of $D$. In other words, there exists a permutation $\pi$ on $\{1,\dots,n\}$ such that $\hat{D}_{ii} = D_{\pi(i) \pi(i)}$. Therefore, we have $(D_{\pi(i) \pi(i)} - D_{jj}) Q_{ij} = 0$ for any $i,j = 1,\dots,n$. Since $D_{11},\dots, D_{nn}$ are distinct, $Q_{ij} = 0$ unless $\pi (i) = j$, in which case $Q_{ij} = 1$ due to the orthogonality. This completes the proof.  
\end{proof}

\section{Main Results}
\label{sec:main}
We are now ready to prove Theorem~\ref{thm:matrix_rearrange}.

\begin{proof}[Proof of Theorem~\ref{thm:matrix_rearrange}]
Let $A,B \in \mathbb{P}_n$.
Since eigenvalues $\lambda_i(\cdot)$ are continuous on $\mathbb{R}^{n\times n}$ (see, \eg, \cite[Corollary~VI.1.6]{bhatia1997matrix}) and $f$ is continuous on $\mathbb{R}_{++}$, we can also assume without loss of generality that the eigenvalues of $A,B$ are all distinct. Furthermore, suppose that the inequality holds for any function $f$ such that $s\mapsto sf'(s)$ is strictly increasing on $\RR_{++}$. Then, for any $\tilde{f}$ such that $s\mapsto s \tilde{f}'(s)$ is only monotonically increasing on $\RR_{++}$, we consider the perturbed function $\tilde{f}_\epsilon (s) \coloneqq \tilde{f}(s) + \epsilon s$. For $s_2 > s_1 > 0$ and $\epsilon > 0$,
\begin{equation*}
s_2 \tilde{f}_\epsilon' (s_2) - s_1 \tilde{f}_\epsilon' (s_1) = s_2 \tilde{f}'(s_2) - s_1 \tilde{f}'(s_1) + \epsilon (s_2 - s_1) > 0.
\end{equation*}
Therefore, $s\mapsto s \tilde{f}_\epsilon'(s)$ is strictly increasing on $\RR_{++}$. By supposition, the inequality holds for $f = \tilde{f}_\epsilon$. Taking limit $\epsilon \searrow 0$, the rearrangement inequality then holds for $\tilde{f}$ as well. Hence, it suffices to prove the inequality for functions $f$ such that $s\mapsto sf'(s)$ is strictly increasing on $\RR_{++}$.

Since $B^\half A B^{\half} \in \mathbb{P}_n$, by the definition of $S_f$, we have 
\[S_f ( B^{\half} A B^{\half} ) = \sum_{i = 1}^n f( \lambda_i( B^{\half} A B^{\half} )) .\]
We start with the lower bound.
Let $U_A\Sigma_AU_A^\top$ and $U_B\Sigma_BU_B^\top$ be the eigenvalue decompositions of $A$ and $B$, respectively. Consider the minimization problem
\begin{equation}\label{opt:S_f}
\inf_{U\in \mathbb{O}_n}  S_f \big(\Sigma_B^{\half} U\Sigma_A U^\top \Sigma_B^{\half}\big).
\end{equation}
By the continuity of $f$ and eigenvalues $\lambda_i(\cdot)$, the function $U \mapsto S_f (\Sigma_B^{\half} U\Sigma_A U^\top \Sigma_B^{\half} )$ is continuous on $\mathbb{O}_n$. Since $\mathbb{O}_n$ is compact, a minimizer $Q\in \mathbb{O}_n$ to problem~\eqref{opt:S_f} exists. We thus have
\begin{equation*}
S_f ( B^{\half} A B^{\half} ) \geq \min_{U\in \mathbb{O}_n} S_f \big(\Sigma_B^{\half} U\Sigma_A U^\top \Sigma_B^{\half}\big) = S_f\big(\Sigma_B^{\half} Q\Sigma_A Q^\top \Sigma_B^{\half}\big).
\end{equation*}
Let 
\begin{equation*}
\hat{A} = Q \Sigma_A Q^\top \quad\text{and}\quad C = \Sigma_B^{\half}\hat{A}\Sigma_B^{\half}.
\end{equation*}
Since $A, B \in\mathbb{P}_n$, we have that $C  \in\mathbb{P}_n$ and hence that $\lambda (C) = \sigma (C)$.
Fix any eigenvalue decomposition $C = U_C \Diag(\lambda(C)) U_C^\top$ and let $\Delta = U_C \Diag\left( f' (\lambda(C)) \right)U_C^\top$.
By Proposition~\ref{prop:pert}, $
\Delta$ is the derivative of $S_f$ at $C$. Consider the skew-symmetric matrix \[
K = \hat{A}\Sigma_B^{\half}\Delta \Sigma_B^{\half} - \Sigma_B^{\half} \Delta \Sigma_B^{\half}\hat{A}.\] 
By the skew-symmetry of $K$, $ \mathrm{Exp} (\epsilon K)\in \mathbb{O}_n$ for any $\epsilon\in\mathbb{R}$, where $ \mathrm{Exp} (\cdot)$ denotes the matrix exponential; see~\cite[Section~1]{cardoso2010exponentials} for example. 
Recalling that $\mathrm{Exp}(X) = I_n + X + \tfrac{1}{2}X^2 + \cdots$ for any matrix $X\in\mathbb{R}^{n\times n}$, we have 
\begin{equation*}
\begin{split}
& S_f \big(\Sigma_B^{\half} \mathrm{Exp} (\epsilon K)\hat{A} \mathrm{Exp} (\epsilon K)^\top \Sigma_B^{\half} \big) \\
= & S_f \big(\Sigma_B^{\half} (I+\epsilon K)\hat{A}(I-\epsilon K) \Sigma_B^{\half}\big) + o(\epsilon) \\
= & S_f \big( \Sigma_B^{\half}\hat{A} \Sigma_B^{\half} + \epsilon \Sigma_B^{\half}(K\hat{A}-\hat{A} K ) \Sigma_B^{\half} \big) + o(\epsilon) \\
= & S_f \Big( \Sigma_B^{\half}\hat{A}\Sigma_B^{\half} \Big) + \epsilon\left\langle \Sigma_B^{\half} K \hat{A}\Sigma_B^{\half} - \Sigma_B^{\half}\hat{A} K \Sigma_B^{\half},\, \Delta  \right\rangle + o(\epsilon) \\
= & S_f \Big( \Sigma_B^{\half}\hat{A}\Sigma_B^{\half}\Big) + \epsilon\left\langle  K ,\, \Sigma_B^{\half} \Delta \Sigma_B^{\half}\hat{A} - \hat{A}\Sigma_B^{\half} \Delta \Sigma_B^{\half}\right\rangle + o(\epsilon) \\
= & S_f \Big( \Sigma_B^{\half}\hat{A}\Sigma_B^{\half} \Big)-\epsilon\norm{ K }^2_{\mathrm{F}} + o(\epsilon),
\end{split}
\end{equation*}
where $\norm{\,\cdot\,}_{\mathrm{F}}$ denotes the Frobenius norm, the first and second equalities follow from the continuity of $f$ and eigenvalues $\lambda_i(\cdot)$, the third from Proposition~\ref{prop:pert}, the fourth by the symmetry of $\Sigma_B^{\half}$ and $\hat{A}$, and the fifth by the definition of $K$.
Therefore, $ K = 0$ because otherwise the last display would violate the minimality of $Q$ by taking a sufficiently small $\epsilon>0$. Hence, we have that
\begin{equation*}
\hat{A} \Sigma_B^{\half} \Delta  \Sigma_B^{\half} =\Sigma_B^{\half} \Delta  \Sigma_B^{\half}\hat{A}, 
\end{equation*}
which, upon multiplying both sides by $\Sigma_B^{\half}$, yields
\begin{equation}\label{eq:pf_1}
C \Delta  \Sigma_B = \Sigma_B \Delta C.
\end{equation}
Letting
\begin{equation}\label{def:C_hat}
\hat{C} = U_C \Diag\left(\lambda(C) \circ f'(\lambda(C)) \right)U_C^\top ,
\end{equation}
it then follows from the definition of $\Delta$ that
\begin{equation*}
C \Delta 
=  U_C \Diag\left(\lambda(C) \circ f'(\lambda(C)) \right)U_C^\top = \hat{C} = U_C \Diag\left(f'(\lambda(C) \circ \lambda(C)) \right)U_C^\top = \Delta C.
\end{equation*}
Thus, equality~\eqref{eq:pf_1} is equivalent to
\begin{equation}\label{eq:7}
\hat{C} \Sigma_B = \Sigma_B \hat{C}.
\end{equation}
In other words, $\Sigma_B$ commutes with $\hat{C}$.
We claim that $\Sigma_B$ also commutes with $C$. To prove the claim, note that we can write
\begin{equation}\label{eq:diag_form}
\Diag(\lambda(C)) = 
\begin{pmatrix}
c_1 I_{n_1} & & \\
 & \ddots & \\
 & & c_\ell I_{n_\ell}
\end{pmatrix},
\end{equation}
for some positive integers $\ell, n_1,\dots,n_\ell$ with $n_1 +\cdots + n_\ell = n$ and real numbers $c_1>\cdots> c_\ell >0 $.
Since the function $s\mapsto s f'(s)$ is strictly increasing on $\RR_{++}$,
the diagonal matrix $ \Diag\left(\lambda(C) \circ f'(\lambda(C)) \right)$ takes the same form as \eqref{eq:diag_form}, \ie, for some real numbers $t_1 > \cdots > t_\ell >0 $,
\begin{equation}\label{eq:diag_form_1}
\Diag\left(\lambda(C) \circ f'(\lambda(C)) \right) = 
\begin{pmatrix}
t_1 I_{n_1} & & \\
 & \ddots & \\
 & & t_\ell I_{n_\ell}
\end{pmatrix}.
\end{equation}
From~\eqref{def:C_hat}, \eqref{eq:7} and \eqref{eq:diag_form_1}, we have that 
\begin{equation*}
U_C^\top\Sigma_B U_C\begin{pmatrix}
t_1I_{n_1}&&\\
&\ddots&\\
&&t_\ell I_{n_\ell}\\
\end{pmatrix}=\begin{pmatrix}
t_1I_{n_1}&&\\
&\ddots&\\
&&t_\ell I_{n_\ell}\\
\end{pmatrix}U_C^\top\Sigma_B U_C,
\end{equation*}
which, upon invoking Lemma~\ref{lem:diag_commute}, implies the existence of symmetric matrices $\tilde{B}_1\in\mathbb{R}^{n_1\times n_1},\dots, \tilde{B}_\ell \in \mathbb{R}^{n_\ell \times n_\ell}$ such that
\begin{equation*}
U_C^\top\Sigma_B U_C=\begin{pmatrix}
\tilde{B}_1&&\\
&\ddots&\\
&&\tilde{B}_\ell
\end{pmatrix}.
\end{equation*}
Observing that
\begin{equation*}
\begin{pmatrix}
\tilde{B}_1&&\\
&\ddots&\\
&&\tilde{B}_\ell
\end{pmatrix}\begin{pmatrix}
c_1I_{n_1}&&\\
&\ddots&\\
&&c_\ell I_{n_\ell }
\end{pmatrix} = \begin{pmatrix}
c_1I_{n_1}&&\\
&\ddots&\\
&&c_\ell I_{n_\ell }
\end{pmatrix}\begin{pmatrix}
\tilde{B}_1&&\\
&\ddots&\\
&&\tilde{B}_\ell 
\end{pmatrix},
\end{equation*}
we arrive at
\begin{equation*}
U_C^\top \Sigma_B U_C \Diag(\lambda(C)) = \Diag(\lambda(C)) U_C^\top \Sigma_B U_C .
\end{equation*}
Multiplying the last display by $U_C$ from the left and $U_C^\top$ from the right, we get
\begin{equation}\label{eq:commute}
\Sigma_B C = C \Sigma_B, 
\end{equation}
which proves the claim. Then, it follows from equality~\eqref{eq:commute} and the definition of $C$ that $\Sigma_B \hat{A} = \hat{A} \Sigma_B$. 
Since $\Sigma_B$ is a diagonal matrix with distinct diagonal entries, the matrix $\hat{A} = Q \Sigma_A Q^\top$ is also diagonal by Lemma~\ref{lem:diag_commute}. Next, using Lemma~\ref{lem:orthogonal_conjugation} and that $\Sigma_A$ is a diagonal matrix with distinct diagonal entries, the minimizer $Q$ is a permutation matrix.
Hence, there exists a permutation $\pi$ on $\{1,\dots,n\}$ such that
\begin{equation}\label{eq:6}
S_f \big(\Sigma_B^{\half} Q \Sigma_A Q^\top \Sigma_B^{\half} \big) = \sum_{i = 1}^n f \left( \lambda_i (A) \lambda_{\pi(i)}(B) \right).
\end{equation}
Using \eqref{eq:6}, Lemma~\ref{lem:rearrange} and the fact that $A,B \in\mathbb{P}_n$, we get
\begin{equation*}
S_f ( B^{\half} A B^{\half} ) \ge S_f \big( \Sigma_B^{\half} Q \Sigma_A Q^\top \Sigma_B^{\half} \big) \ge \sum_{i = 1}^n f(\lambda_i(A) \lambda_{n-i+1}(B)) ,
\end{equation*}
which yields the lower bound. The upper bound of $S_f(B^{\half} A B^{\half})$ can be proved similarly by considering the maximization problem
\begin{equation*}
\sup_{U\in \mathbb{O}_n} S_f  \big(\Sigma_B^{\half} U \Sigma_A U^\top \Sigma_B^{\half} \big),
\end{equation*}
instead of the minimization problem~\eqref{opt:S_f}. Finally, the extension to positive semidefinite matrices follows from the right continuity of $f$ at $0$, the continuity of eigenvalues $\lambda_i(\cdot)$ and taking limits. This completes the proof.
\end{proof}

We next prove a matrix rearrangement inequality for singular values of rectangular matrices.
\begin{thm}\label{thm:matrix_rearrange_singular}
Let $f:\mathbb{R}_{++}\ra \mathbb{R}$ be a differentiable function such that $s\mapsto sf'(s)$ is monotonically increasing on $\RR_{++}$. Then, for any full-rank matrices $X,Y\in  \RR^{m\times n}$,
\begin{equation*}
\sum_{i=1}^{\min\{m,n\}} f\left(\sigma_i (X)\sigma_{n-i+1} (Y)\right) \le \sum_{i=1}^{\min\{m,n\}} f\big(\sigma_i (X^\top Y)\big)\le \sum_{i=1}^{\min\{m,n\}} f\left(\sigma_i (X)\sigma_i (Y)\right).
\end{equation*}
If $f$ is additionally defined and right-continuous at $0$, then the inequality holds for any matrices $X,Y\in  \RR^{m\times n}$.
\end{thm}

\begin{proof}
    Let $X,Y\in \mathbb{R}^{m\times n}$. Since $\sigma_i ( X^\top Y ) = \sigma_i (X Y^\top)$ for $i = 1,\dots, \min\{ m, n\}$, we can assume without loss of generality that $m \le n$.
    By the definition of singular values, we have that for any $i = 1,\dots, m$, 
    \begin{align*}
        \sigma_i(X^\top Y ) = \lambda_i^{\half} ( X^\top Y Y^\top X ) = \lambda_i^{\half} ( X X^\top Y Y^\top ) .
    \end{align*}
    Similarly, we have that for any $i = 1,\dots, m$, 
    \begin{align*}
        \sigma_i (X) = \lambda_i^{\half} (XX^\top) \quad \text{and}\quad \sigma_i (Y) = \lambda_i^{\half} (YY^\top).
    \end{align*}
    Also, $ XX^\top$ and $YY^\top$ are positive definite if and only if $X$ and $Y$ have full rank. Next, let $g: \mathbb{R}_{++} \to \mathbb{R}$ be the function defined by $g(t) = f(\sqrt{t})$. Then,  $g$ is differentiable on $\mathbb{R}_{++}$ and $t g'(t) =  \tfrac{1}{2} \sqrt{t} f'(\sqrt{t})$, whose monotonicity inherits from the map $s\mapsto s f'(s)$. Moreover, if $f$ is defined and right-continuous at $0$, so is $g$.
    Noting that $\lambda_i ( B^{\half} A B^{\half} ) = \lambda_i (A B)$ for any positive semidefinite matrices $A, B\in \mathbb{R}^{m\times m}$, applying Theorem~\ref{thm:matrix_rearrange} with $A = XX^\top$ and $B = YY^\top$ yields the desired conclusion.
\end{proof}

\section{Applications}
\label{sec:applications}

\subsection{Schatten Quasi-Norms}
\label{subsec:schatten}
For $q > 0$ and $X\in\mathbb{R}^{m\times n}$, we denote 
\begin{equation*}
\norm{X}_q = \left( \sum_{i = 1}^{ \min\{m,n\} } \sigma_i^q (X) \right)^{\frac{1}{q}}.
\end{equation*}
If $q \ge 1$, $\norm{X}_q$ is the so-called Schatten-$q$ norm of $X$. 
The Banach space associated with the Schatten-$q$ norm is a classical subject in operator theory and has attracted much research since the forties, see~\cite{schatten1960norm,gohberg1969introduction,tomczak1974moduli}. On the other hand, if $q \in (0,1)$, $\norm{X}_q$ is no longer a norm but only a quasi-norm.
Motivated by its proximity to the rank function, the Schatten-$q$ quasi-norm with $q\in (0,1]$ has been applied to low-rank matrix recovery~\cite{rohde2011estimation,yue2016perturbation}. 

As an application of our Theorem~\ref{thm:matrix_rearrange_singular}, we obtain the following inequality on $\norm{\cdot}_q$ for general $q\in \mathbb{R}$, which could potentially find applications in the analysis of the statistical properties of and numerical algorithms for low-rank matrix recovery based on the Schatten-$q$ quasi-norm.
\begin{coro}
\label{coro:schatten}
Let $X, Y\in \mathbb{R}^{m\times n}$ and $q > 0$. Then, it holds that
\begin{equation*}
 \sum_{i = 1}^{ \min\{m,n\} } \sigma_i^q (X) \sigma_{n-i+1}^q (Y) \le \big\|X^\top  Y\big\|_q^q \le \sum_{i = 1}^{ \min\{m,n\} } \sigma_i^q (X) \sigma_{i}^q (Y).
\end{equation*}
\end{coro}

\begin{proof}
Let $f(s) = s^q$ for $s\in \mathbb{R}_+$. Then, the function $s\mapsto s f'(s) = q s^q$ is monotonically increasing on $\mathbb{R}_{++}$. The desired inequality then follows from Theorem~\ref{thm:matrix_rearrange_singular}.
\end{proof}

\subsection{Affine-Invariant Geometry on $\mathbb{P}_n$}
\label{subsec:FR_RG}

It is well-known that the cone $\mathbb{P}_n$ of $n\times n$ positive definite matrices is a differentiable manifold of dimension $n(n+1)/2$, see, \eg, \cite[Chapter 6]{bhatia2009positive}. For any $A \in \mathbb{P}_n$, the tangent space $T_A \mathbb{P}_n$ at $A$ can be identified with the set of $n\times n$ symmetric matrices $\mathbb{S}_n$. We can equip the cone $\mathbb{P}_n$ with a Riemannian metric called the affine-invariant metric: for any $X, Y \in T_A \mathbb{P}_n \cong \mathbb{S}_n$,
\begin{equation*}
\langle X , Y\rangle_A \coloneqq \Tr\big( X A^{-1} Y A^{-1} \big).
\end{equation*}
Indeed, one can easily check that, given any $A \in\mathbb{P}_n$, the map $\langle \cdot , \cdot \rangle_A$ defines a symmetric positive definite bilinear form on $\mathbb{S}_n$. For any $A,B\in \mathbb{P}_n$, The corresponding Riemannian distance is given by
\begin{equation*}
d_{\mathbb{P}_n}  (A, B) = \big\|\mathrm{Log} \big(B^{-\half} A B^{-\half}\big) \big\|_{\mathrm{F}}, 
\end{equation*}
where $\mathrm{Log}(\,\cdot\, )$ denotes the matrix logarithm. This distance enjoys many interesting properties~\cite[Chapter 6]{bhatia2009positive} and finds applications in diverse areas such as machine learning~\cite{nguyen2019calculating,sra2016geometric}, image and video processing~\cite{dryden2009non,tuzel2008pedestrian} and elasticity theory~\cite{moakher2006closest,neff2014logarithmic}. 

More generally, for any $q\ge 1$ and $A,B\in \mathbb{P}_n$, we define 
\begin{equation*}
d_q (A, B) = \big\| \mathrm{Log} \big( B^{-\half} A B^{-\half} \big) \big\|_q.
\end{equation*} 
It has been proved in \cite[Section 6]{bhatia2009positive} that $d_q $ is a distance on $\mathbb{P}_n$ for any $q \ge 1$. 
\begin{coro}\label{coro:d_q}
Let $A,B\in \mathbb{P}_n$ and $q \ge 1$. Then, it holds that
\begin{equation*}
\sum_{i = 1}^n \left| \log  \lambda_i (A) -\log \lambda_i (B)  \right|^q \le d^q_q (A, B) \le  \sum_{i = 1}^n \left| \log  \lambda_i (A) -\log \lambda_{n-i+1} (B)  \right|^q .
\end{equation*}
\end{coro}

\begin{proof}
We first assume that $q> 1$. Next, we note that
\begin{equation*}
 d_q^q (A,B)  = \norm{\mathrm{Log}\big( B^{-\half} A B^{-\half} \big)}_q^q = \sum_{ i = 1}^n f\big( \lambda_i \big( B^{-\half} A B^{-\half} \big) \big),
\end{equation*}
where $f (s) = |\log s|^q$ for $s \in \mathbb{R}_{++}$. It can be readily verified that $f$ is differentiable on $\mathbb{R}_{++}$ and 
\begin{equation*}
f' (s) = \begin{cases}
\frac{q}{s} (\log s)^{q-1}, & \text{if } s\ge 1,\\
-\frac{q}{s} (\log \frac{1}{s})^{q-1}, & \text{if } 0<s< 1.
\end{cases}
\end{equation*}
Hence,
\begin{equation}\label{eq:sf'}
s f'(s) = \sgn(\log s)\cdot q |\log s|^{q-1},
\end{equation}
where $\sgn(\cdot)$ denotes the sign of a real number.
To show that $s\mapsto s f' (s)$ is monotonically increasing, we let $ s_2 > s_1 > 0$ and consider four different cases: $s_2 \ge 1> s_1 > 0 $, $s_2 > 1\ge s_1 > 0 $, $s_2 > s_1 \ge 1$ and $1 \ge s_2 > s_1 > 0$. For the first two cases, by \eqref{eq:sf'}, we have that $s_2 f' (s_2) \ge 0 \ge s_1 f' (s_1)$.
For the third case of $s_2 > s_1 \ge 1$, \eqref{eq:sf'} shows that the function $ s\mapsto sf'(s) $ is continuous on $[1,\infty)$ and differentiable on $(1,\infty)$. Also, for any $s > 1$,
\begin{equation*}
(sf'(s))' = \left( q (\log s)^{q-1} \right)' = \frac{q(q-1) (\log s)^{q-2}}{s} \ge 0,
\end{equation*} 
which implies that the function $s\mapsto sf'(s)$ is monotonically increasing on $[1,\infty)$.
We therefore have $s_2 f' (s_2) \ge s_1 f' (s_1)$.
Similarly, for the fourth case of $s_1 < s_2 \le 1$, \eqref{eq:sf'} shows that the function $ s\mapsto sf'(s) $ is continuous on $(0, 1]$ and differentiable on $(0,1)$. 
Also, for any $s \in (0, 1)$,
\begin{equation*}
(sf'(s))' = \left(- q \left(\log \tfrac{1}{s} \right)^{q-1} \right)' = \frac{q(q-1) \left(\log\tfrac{1}{s} \right)^{q-2}}{s} \ge 0,
\end{equation*}
which implies that the function $s\mapsto sf'(s)$ is monotonically increasing on $(0, 1]$.
We therefore have $s_2 f' (s_2) \ge s_1 f' (s_1)$.
Hence, by Theorem~\ref{thm:matrix_rearrange}, we obtain
\begin{equation*}
\sum_{i=1}^n f\big(\lambda_i (A)\lambda_i (B^{-1})\big) \ge \sum_{ i = 1}^n f\big( \lambda_i \big( B^{-\half} A B^{-\half} \big) \big) \ge \sum_{i=1}^n f\big(\lambda_i (A)\lambda_{n-i+1} (B^{-1})\big),
\end{equation*}
which is equivalent to
\begin{equation*}
\sum_{i = 1}^n \left| \log  \lambda_i (A) -\log \lambda_i (B)  \right|^q \le d^q_q (A, B) \le  \sum_{i = 1}^n \left| \log  \lambda_i (A) -\log \lambda_{n-i+1} (B)  \right|^q .
\end{equation*}
The case of $q = 1$ follows from limiting arguments. This completes the proof.
\end{proof}

For any $A,B\in \mathbb{P}_n$, Corollary~\ref{coro:d_q} with $q = 2$ immediately implies that the Riemannian distance $d_{\mathbb{P}_n}$ with respect to the affine-invariant metric satisfies the inequality
\begin{equation*}\label{ineq:d_P}
\sum_{i = 1}^n \left( \log  \lambda_i (A) -\log \lambda_i (B)  \right)^2 \le d^2_{\mathbb{P}_n} (A, B) \le  \sum_{i = 1}^n \left( \log  \lambda_i (A) -\log \lambda_{n-i+1} (B)  \right)^2 .
\end{equation*}

\subsection{Alpha-Beta Log-Determinant Divergences}
\label{subsec:ABLD}
Divergences, which are measures of dissimilarity between two positive definite matrices, play an important role in information geometry and find applications across many areas, see \cite{nielsen2013matrix,cichocki2015log,amari2016information} and the references therein.  As a unification and generalization of many existing divergences in the literature, the family of Alpha-Beta log-determinant divergences (or AB log-det divergences for short) is introduced and studied in \cite{cichocki2015log}. Given any $\alpha, \beta\in \mathbb{R}$ such that $\alpha \beta \neq 0$ and $\alpha + \beta \neq 0$, the AB log-det divergence with parameter $\alpha$ and $\beta$ between $A,B\in \mathbb{P}_n$ is defined as 
\begin{equation*}
D_{\alpha, \beta}(A \| B) = \frac{1}{\alpha \beta} \log \det\left( \frac{\alpha ( A B^{-1})^\beta + \beta (A B^{-1})^{-\alpha}  }{\alpha + \beta} \right)
\end{equation*}
The definition of the AB log-det divergence can be extended to the cases of $\alpha \beta= 0$ and/or $\alpha + \beta = 0$ by taking limits. In particular, we have
\begin{equation*}
D_{\alpha, \beta}(A \| B) =
\begin{cases}
\medskip
\displaystyle\frac{1}{\alpha^2} \left( \Tr \big(\big(BA^{-1}\big)^\alpha - I\big) - \alpha \log\det\big(BA^{-1}\big)\right), & \text{if } \alpha\neq 0,\, \beta = 0, \\

\medskip
\displaystyle\frac{1}{\beta^2} \left( \Tr \big( \big(AB^{-1}\big)^\beta - I\big) - \beta \log\det\big(AB^{-1}\big)\right), & \text{if } \beta\neq 0,\, \alpha = 0, \\

\medskip
\displaystyle\frac{1}{\alpha^2} \log\left( \frac{ \det \big(AB^{-1}\big)^{\alpha} }{ \det \left( I + \log\big(AB^{-1}\big)^\alpha  \right) } \right), & \text{if } \alpha = -\beta \neq 0.
\end{cases}
\end{equation*}

For $\alpha = \beta = 0$, $D_{0,0} (A\| B) = \tfrac{1}{2} d_{\mathbb{P}_n}^2 (A,B)$. We thus omit the discussion on this case and refer the readers to Section~\ref{subsec:FR_RG}.
Besides the squared affine-invariant Riemannian metric, many other well-known divergences are special cases of AB log-det divergences, including the S-divergence~\cite{sra2016positive} where $\alpha = \beta = \tfrac{1}{2}$ and the Stein's loss~\cite{james1992estimation} (also called the Burg divergence) where $\alpha = 0$ and $\beta = 1$. For more examples of AB log-det divergences, we refer the readers to \cite[Section 3]{cichocki2015log}.

Note that $AB^{-1}$ is diagonalizable for any $A,B\in\mathbb{P}_n$. Therefore, as pointed out in \cite{cichocki2015log}, AB log-det divergences $D_{\alpha, \beta}(A \| B)$ can  be expressed via the (positive) eigenvalues of the matrix $AB^{-1}$, which coincide with those of the matrix $B^{-\half} A B^{-\half}$:
\begin{equation*}
D_{\alpha, \beta}(A \| B) = 
\begin{cases}
\medskip
\displaystyle\frac{1}{\alpha \beta} \sum_{i=1}^n \log\left( \frac{ \alpha \lambda_i^\beta (B^{-\half} A B^{-\half}) + \beta \lambda_i^{-\alpha} (B^{-\half} A B^{-\half}) }{\alpha + \beta} \right), &\text{if } \alpha \beta,\, \alpha + \beta\neq 0, \\

\medskip
\displaystyle\frac{1}{\alpha^2} \sum_{i=1}^n \left( \lambda_i^{-\alpha} (B^{-\half} A B^{-\half}) - \log\lambda_i^{-\alpha} (B^{-\half} A B^{-\half}) -1 \right) , & \text{if } \alpha\neq 0,\, \beta = 0, \\

\medskip
\displaystyle\frac{1}{\beta^2} \sum_{i=1}^n \left( \lambda_i^{\beta} (B^{-\half} A B^{-\half}) - \log\lambda_i^{\beta} (B^{-\half} A B^{-\half}) -1 \right), & \text{if } \beta\neq 0,\, \alpha = 0, \\

\medskip
\displaystyle\frac{1}{\alpha^2} \sum_{i=1}^n \log \left( \frac{\lambda_i^\alpha \left( B^{-\half} A B^{-\half} \right) }{1 + \log \lambda_i^\alpha \left( B^{-\half} A B^{-\half} \right)} \right), & \text{if } \alpha = -\beta \neq 0.
\end{cases}
\end{equation*}

The following upper and lower bounds for AB log-det divergences are generalizations of the \cite[Corollary 3.8]{sra2016positive}.

\begin{coro}
Let $A,B\in \mathbb{P}_n$ and $\alpha, \beta \in \mathbb{R}$. Then,
\begin{equation*}
D_{\alpha, \beta}(A \| B) \le
\begin{cases}
\medskip
\displaystyle\frac{1}{\alpha \beta} \sum_{i=1}^n \log\left( \frac{ \alpha \lambda_i^\beta (A) \lambda_{n-i+1}^{-\beta} (B) + \beta \lambda_i^{-\alpha} (A) \lambda_{n-i+1}^{\alpha} (B) }{\alpha + \beta} \right), &\text{if } \alpha \beta >0,\, \alpha + \beta\neq 0, \\

\medskip
\displaystyle\frac{1}{\alpha \beta} \sum_{i=1}^n \log\left( \frac{ \alpha \lambda_i^\beta (A) \lambda_{i}^{-\beta} (B) + \beta \lambda_i^{-\alpha} (A) \lambda_{i}^{\alpha} (B) }{\alpha + \beta} \right), &\text{if } \alpha \beta < 0,\, \alpha + \beta\neq 0, \\

\medskip
\displaystyle\frac{1}{\alpha^2} \sum_{i=1}^n \left( \frac{\lambda_{n-i+1}^{\alpha} (B)}{\lambda_i^{\alpha} (A)}  - \log\left( \frac{\lambda_{n-i+1}^{\alpha} (B)}{\lambda_i^{\alpha} (A)} \right) -1 \right) , & \text{if } \alpha\neq 0,\, \beta = 0, \\

\medskip
\displaystyle\frac{1}{\beta^2} \sum_{i=1}^n \left( \frac{\lambda_i^{\beta} (A)}{\lambda_{n-i+1}^{\beta} (B)}  - \log\left( \frac{\lambda_i^{\beta} (A)}{\lambda_{n-i+1}^{\beta} (B)} \right) -1 \right), & \text{if } \beta\neq 0,\, \alpha = 0, \\

\medskip
\displaystyle\frac{1}{\alpha^2} \sum_{i=1}^n \log \left( \frac{\lambda_i^\alpha (A) \lambda_{n-i+1}^{-\alpha} (B) }{1 + \log\left( \lambda_i^\alpha (A) \lambda_{n-i+1}^{-\alpha} (B) \right) } \right), & \text{if } \alpha = -\beta \neq 0;
\end{cases}
\end{equation*}
and 
\begin{equation*}
D_{\alpha, \beta}(A \| B) \ge
\begin{cases}
\medskip
\displaystyle\frac{1}{\alpha \beta} \sum_{i=1}^n \log\left( \frac{ \alpha \lambda_i^\beta (A) \lambda_{i}^{-\beta} (B) + \beta \lambda_i^{-\alpha} (A) \lambda_{i}^{\alpha} (B) }{\alpha + \beta} \right), &\text{if } \alpha \beta >0,\, \alpha + \beta\neq 0, \\

\medskip
\displaystyle\frac{1}{\alpha \beta} \sum_{i=1}^n \log\left( \frac{ \alpha \lambda_i^\beta (A) \lambda_{n-i+1}^{-\beta} (B) + \beta \lambda_i^{-\alpha} (A) \lambda_{n-i+1}^{\alpha} (B) }{\alpha + \beta} \right), &\text{if } \alpha \beta < 0,\, \alpha + \beta\neq 0, \\

\medskip
\displaystyle\frac{1}{\alpha^2} \sum_{i=1}^n \left( \frac{\lambda_{i}^{\alpha} (B)}{\lambda_i^{\alpha} (A)}  - \log\left( \frac{\lambda_{i}^{\alpha} (B)}{\lambda_i^{\alpha} (A)} \right) -1 \right) , & \text{if } \alpha\neq 0,\, \beta = 0, \\

\medskip
\displaystyle\frac{1}{\beta^2} \sum_{i=1}^n \left( \frac{\lambda_{i}^{\beta} (A)}{\lambda_i^{\beta} (B)}  - \log \left( \frac{\lambda_{i}^{\beta} (A)}{\lambda_i^{\beta} (B)} \right) -1 \right), & \text{if } \beta\neq 0,\, \alpha = 0, \\

\medskip
\displaystyle\frac{1}{\alpha^2} \sum_{i=1}^n \log \left( \frac{\lambda_i^\alpha (A) \lambda_{i}^{-\alpha} (B) }{1 + \log\left( \lambda_i^\alpha (A) \lambda_{i}^{-\alpha} (B) \right) } \right), & \text{if } \alpha = -\beta \neq 0.
\end{cases}
\end{equation*}
\end{coro}

\begin{proof}
We start with the case of $\alpha \beta, \alpha + \beta\neq 0$. Consider the function $f: \mathbb{R}_{++} \ra \mathbb{R} $ defined by
\begin{equation*}
f(s) = \log \left( \frac{\alpha s^\beta + \beta s^{-\alpha}}{\alpha + \beta} \right).
\end{equation*}
We have that for any $s > 0$,
\begin{equation*}
s f' (s) = \frac{\alpha \beta (s^\beta - s^{-\alpha})}{\alpha s^\beta + \beta s^{-\alpha}}.
\end{equation*}
Then, for any $s > 0$,
\begin{equation*}
\left( s f' (s) \right)' = \frac{\alpha \beta (\alpha + \beta)^2 s^{ \alpha + \beta - 1}}{(\alpha s^{\alpha + \beta} + \beta )^2} =
\begin{cases}
> 0, &\text{if } \alpha\beta >0, \\
< 0, &\text{if } \alpha\beta <0.
\end{cases}
\end{equation*}
By Theorem~\ref{thm:matrix_rearrange}, if $\alpha \beta >0$, then
\begin{align*}
& \frac{1}{\alpha \beta} \sum_{i=1}^n \log\left( \frac{ \alpha \lambda_i^\beta (A) \lambda_{i}^{-\beta} (B) + \beta \lambda_i^{-\alpha} (A) \lambda_{i}^{\alpha} (B) }{\alpha + \beta} \right) \\
& \quad\le D_{\alpha, \beta}(A \| B) \le \frac{1}{\alpha \beta} \sum_{i=1}^n \log\left( \frac{ \alpha \lambda_i^\beta (A) \lambda_{n-i+1}^{-\beta} (B) + \beta \lambda_i^{-\alpha} (A) \lambda_{n-i+1}^{\alpha} (B) }{\alpha + \beta} \right) ;
\end{align*}
and if $\alpha \beta <0$, then 
\begin{align*}
& \frac{1}{\alpha \beta} \sum_{i=1}^n \log\left( \frac{ \alpha \lambda_i^\beta (A) \lambda_{i}^{-\beta} (B) + \beta \lambda_i^{-\alpha} (A) \lambda_{i}^{\alpha} (B) }{\alpha + \beta} \right) \\
& \quad\ge D_{\alpha, \beta}(A \| B) \ge \frac{1}{\alpha \beta} \sum_{i=1}^n \log\left( \frac{ \alpha \lambda_i^\beta (A) \lambda_{n-i+1}^{-\beta} (B) + \beta \lambda_i^{-\alpha} (A) \lambda_{n-i+1}^{\alpha} (B) }{\alpha + \beta} \right) .
\end{align*}
For the case of $\alpha\neq 0$ and $\beta = 0$, we consider the function $f: \mathbb{R}_{++} \ra \mathbb{R} $ defined by 
\begin{equation*}
f(s) = s^{-\alpha} + \alpha \log s   -1 .
\end{equation*}
We have that for any $s > 0$,
\begin{equation*}
s f'(s) = \alpha (1 -s^{-\alpha}) ,
\end{equation*}
which is monotonically increasing regardless of the sign of $\alpha$.
By Theorem~\ref{thm:matrix_rearrange}, 
\begin{align*}
& \frac{1}{\alpha^2} \sum_{i=1}^n \left( \frac{\lambda_{i}^{\alpha} (B)}{\lambda_i^{\alpha} (A)}  - \log \left( \frac{\lambda_{i}^{\alpha} (B)}{\lambda_i^{\alpha} (A)} \right) -1 \right) \\
& \quad\le D_{\alpha, \beta}(A \| B) \le \frac{1}{\alpha^2} \sum_{i=1}^n \left( \frac{\lambda_{n-i+1}^{\alpha} (B)}{\lambda_i^{\alpha} (A)}  - \log\left( \frac{\lambda_{n-i+1}^{\alpha} (B)}{\lambda_i^{\alpha} (A)} \right) -1 \right) .
\end{align*}
For the case of $\beta\neq 0$ and $\alpha = 0$, by using exactly the same argument as that for the case of $\alpha\neq 0$ and $\beta = 0$, we can prove that
\begin{align*}
& \frac{1}{\beta^2} \sum_{i=1}^n \left( \frac{\lambda_{i}^{\beta} (A)}{\lambda_i^{\beta} (B)}  - \log \left( \frac{\lambda_{i}^{\beta} (A)}{\lambda_i^{\beta} (B)} \right) -1 \right) \\
& \quad\le D_{\alpha, \beta}(A \| B) \le \frac{1}{\beta^2} \sum_{i=1}^n \left( \frac{\lambda_i^{\beta} (A)}{\lambda_{n-i+1}^{\beta} (B)}  - \log \left( \frac{\lambda_i^{\beta} (A)}{\lambda_{n-i+1}^{\beta} (B)} \right) -1 \right) .
\end{align*}
For the case of $\alpha = -\beta \neq 0$, we consider the function $f: \mathbb{R}_{++} \ra \mathbb{R} $ defined by 
\begin{equation*}
f(s) = \log\left( \frac{s^{\alpha}}{1+ \alpha \log s} \right).
\end{equation*}
We have that for any $s > 0$,
\begin{equation*}
s f'(s) = \frac{\alpha^2 \log s}{1+\alpha \log s} .
\end{equation*}
Then, for any $s>0$,
\begin{equation*}
\left( s f' (s) \right)' = \frac{\alpha^2 }{s (1 + \alpha \log s)^2} > 0.
\end{equation*}
By Theorem~\ref{thm:matrix_rearrange}, 
\begin{align*}
& \frac{1}{\alpha^2} \sum_{i=1}^n \log \left( \frac{\lambda_i^\alpha (A) \lambda_{i}^{-\alpha} (B) }{1 + \log\left( \lambda_i^\alpha (A) \lambda_{i}^{-\alpha} (B) \right) } \right) \\
& \quad\le D_{\alpha, \beta}(A \| B) \le \frac{1}{\alpha^2} \sum_{i=1}^n \log \left( \frac{\lambda_i^\alpha (A) \lambda_{n-i+1}^{-\alpha} (B) }{1 + \log\left( \lambda_i^\alpha (A) \lambda_{n-i+1}^{-\alpha} (B) \right) } \right) .
\end{align*}
This completes the proof.
\end{proof}

\section{Conclusion}
This paper generalizes the classical Hardy-Littlewood-P\'olya rearrangement inequality to the matrix setting and presents several applications of the resulting matrix rearrangement inequalities. Rearrangement inequalities have long served as fundamental tools in mathematics, economics, statistics, and signal processing. The present work broadens this scope by establishing new inequalities involving the trace of spectral functions of the product of matrices. A natural direction for future research is to investigate whether Theorem~\ref{thm:matrix_rearrange} admits an extension to the setting of Euclidean Jordan algebras, which is a natural generalization of symmetric matrices endowed with spectral structure.

\section*{Acknowledgements} 
The author is grateful to Cheuk Ting Li, Chi-Kwong Li, Wing-Kin Ma, and Viet Anh Nguyen for their valuable comments on the manuscript. This work is supported in part by the Hong Kong Research Grants Council under the GRF project 15305321.

\bibliographystyle{abbrv}
\bibliography{references}
\end{document}